\documentclass{amsart} 

\usepackage{amsmath,amsthm}     
\usepackage{amssymb}            
\usepackage{euscript}           
\usepackage{enumerate,calc}
\usepackage[matrix,arrow,curve,frame]{xy}    

\xymatrixcolsep{1.9pc}                          
\xymatrixrowsep{1.9pc}
\newdir{ >}{{}*!/-5pt/\dir{>}}                  

\newtheorem{thm}[subsection]{Theorem}
\newtheorem{defn}[subsection]{Definition}
\newtheorem{prop}[subsection]{Proposition}

\newtheorem{cor}[subsection]{Corollary}
\newtheorem{lemma}[subsection]{Lemma}

\newtheorem{ques}[subsection]{Question}

\theoremstyle{definition}  

\newtheorem{exercise}[subsection]{Exercise}

\newtheorem{remark}[subsection]{Remark}

  {\end{list}}
  
\newcommand{\dfn}{\textbf} 

\newcommand{\mdfn}[1]{\dfn{\mathversion{bold}#1}} 

\newcommand{\field}[1]  {\mathbb #1} 

\newcommand{\R}         {\field R}

\newcommand{\HH}        {\field H}
\newcommand{\OO}        {\field O}

\newcommand{\C}         {\field C}

\DeclareMathOperator{\Aut}{Aut}

\DeclareMathOperator{\tr}{tr}

\DeclareMathOperator{\Eig}{Eig}

\newcommand{\norm}[1]{\left| #1 \right|}       

\newcommand{\rea}[1]{|{#1}|}             
\newcommand{\map}{\rightarrow}

\newcommand{\ceck}[1]{\Cech(#1)}         
\newcommand{\oceck}[1]{\Cech^{o}(#1)}    
\newcommand{\oreal}[1]{\rea{\oceck{U}}}  
\newcommand{\creal}[1]{\rea{\ceck{U}}}   

\newcommand{\Cech}{\check{C}}

\newcommand{\llangle}{\langle\langle}
\newcommand{\rrangle}{\rangle\rangle}

\renewcommand{\Re}{\text{Re}}
\renewcommand{\Im}{\text{Im}}

\numberwithin{equation}{subsection}

\begin{document}

\title[Eigentheory of Cayley-Dickson algebras]
{Eigentheory of Cayley-Dickson algebras}

\author[D.~K.~Biss]{Daniel K. Biss}
\author[J.~D.~Christensen]{J.~Daniel Christensen}
\author[D.~Dugger]{Daniel Dugger}
\author[D.~C.~Isaksen]{Daniel C. Isaksen}

\address{Department of Mathematics\\
University of Chicago\\Chicago, IL 60637, USA}

\address{Department of Mathematics\\ 
University of Western Ontario\\
London, Ontario N6A 5B7\\
Canada}

\address{Department of Mathematics\\ University of Oregon\\ Eugene, OR
97403, USA}

\address{Department of Mathematics\\ Wayne State University\\
Detroit, MI 48202, USA}

\email{daniel@math.uchicago.edu}

\email{jdc@uwo.ca}

\email{ddugger@math.uoregon.edu}

\email{isaksen@math.wayne.edu}

\begin{abstract}
We show how eigentheory clarifies many algebraic properties of
Cayley-Dickson algebras.  These notes are intended as background material
for those who are studying this eigentheory more closely.
\end{abstract}

\thanks{MSC: 17A99, 17D99}

\maketitle


\section{Introduction}

Cayley-Dickson algebras are non-associative finite-dimensional
$\R$-algebras that generalize the real numbers, the complex
numbers, the quaternions, and the octonions.  This paper is part of 
a sequence, including \cite{DDD} and \cite{DDDD}, that explores
some detailed algebraic properties of these algebras.

Classically, the first four Cayley-Dickson algebras, i.e.,
$\R$, $\C$, $\HH$, and $\OO$, are viewed as well-behaved, while
the larger Cayley-Dickson algebras are considered to be pathological.
There are several different ways of making this distinction.  One
difference is that the first four algebras do not possess zero-divisors,
while the higher algebras do have zero-divisors.  One of our primary
long-term goals is to understand the zero-divisors in as much detail
as possible; the papers \cite{DDD} and \cite{DDDD} more directly
address this question.  Our motivation for studying zero-divisors
is the potential for useful applications in topology; see
\cite{Co} for more details.

A different but related 
important property of the first four Cayley-Dickson algebras
is that they are {\em alternative}.  This means that $a \cdot ax = a^2 x$
for all $a$ and all $x$.  This is obvious for the associative algebras
$\R$, $\C$, and $\HH$.  It is also true for the octonions.  One 
important consequence of this fact is that it allows for the construction
of the projective line and plane over $\OO$ \cite{B}.

Alternativity fails in the higher Cayley-Dickson algebras; there exist
$a$ and $x$ such that $a \cdot ax$ does not equal $a^2 x$.  
Because alternativity
is so fundamental to the lower Cayley-Dickson algebras, it makes sense
to explore exactly how alternativity fails.  

For various technical reasons that will be apparent later, it turns out
to be inconvenient to consider the operator $L_a^2$, where
$L_a$ is left multiplication by $a$.  Rather, it is preferable to
study the operator $M_a = \frac{1}{\norm{a}^2} L_{a^*} L_a$, 
where $\norm{a}$ is the norm
of $a$ and $a^*$ is the conjugate of $a$.  We will show that $M_a$
is diagonalizable over $\R$.  Moreover, its eigenvalues are non-negative.

Thus we are led to consider the eigentheory of $M_a$.  Given $a$,
we desire to describe the eigenvalues and eigenspaces of $M_a$ in
as much detail as possible.  

This approach to Cayley-Dickson algebras was begun in \cite{MG}.  
However, for completeness, we have reproved
everything that we need here.  

Although the elegance of our results about the eigentheory of $M_a$ speaks
for itself, we give a few reasons why this viewpoint on Cayley-Dickson
algebras is useful.  First, it is possible to completely
classify all subalgebras of the 16-dimensional Cayley-Dickson algebra.
We do not include a proof here because these subalgebras
are classified in \cite{CD}.
On the subject of subalgebras of Cayley-Dickson algebras, 
the article \cite{Ca} is worth noting.

Second, eigentheory supplies one possible solution to the
{\em cancellation problem}.  Namely, given $a$ and $b$, is it possible
to find $x$ such that $ax = b$?  The problem is a technical but
essential idea in \cite[Section 6]{DDDD}.

With alternativity, one can multiply
this equation by $a^*$ on the left and compute that $\norm{a}^2 x = a^*b$.  
Since $\norm{a}^2$ 
is a non-zero real number for any non-zero $a$, this determines
$x$ explicitly.

Now we explain how to solve the equation $ax = b$ without alternativity.
Write $x = \sum x_i$
and $b = \sum b_i$,
where $b_i$ and $x_i$ belong to the $\lambda_i$-eigenspace
of $M_a$.  Multiply on the left by $a^*$ to obtain $a^* \cdot ax = a^*b$,
which can be rewritten as $\norm{a}^2 \sum \lambda_i x_i = \sum a^* b_i$.
As long as none of the eigenvalues $\lambda_i$ are zero, each $x_i$ 
equals $\frac{1}{\lambda_i \norm{a}^2}a^* b_i$,
and therefore $x$ can be recovered.
We expect problems with cancellation
when one of the eigenvalues is zero; this corresponds to the fact that
if $a$ is a zero-divisor, then the cancellation problem might have no
solution or might have non-unique solutions.

We would like to draw the reader's attention to a number of open
questions in Section \ref{sctn:ques}.

\subsection{Conventions}
This paper is not intended to stand independently.  
In particular,
we rely heavily on background from \cite{DDD}.
Section \ref{sctn:CD}
reviews the main points that we will use.

\section{Cayley-Dickson algebras}
\label{sctn:CD}

The \mdfn{Cayley-Dickson algebras} are a sequence of non-associative
$\R$-algebras with involution.  
See \cite{DDD} for a full explanation of the basic properties of
Cayley-Dickson algebras.

These algebras are defined inductively.
We start by defining \mdfn{$A_0$} to be $\R$.  Given $A_{n-1}$, the algebra
\mdfn{$A_n$} is defined additively to be $A_{n-1} \times A_{n-1}$.
Conjugation in $A_n$ is defined by
\[
(a,b)^* = (a^*,-b),
\]
and multiplication is defined by
\[
(a,b)(c,d) = (ac - d^*b, da + bc^*).
\]

One can verify directly from the definitions that $A_1$ is isomorphic
to $\C$; $A_2$ is isomorphic to $\HH$; and $A_3$ is isomorphic to
the octonions $\OO$.  The next algebra $A_4$ is 16-dimensional; it is
sometimes called the {\em hexadecanions} or the {\em sedenions}.

We implicitly view $A_{n-1}$ as the subalgebra $A_{n-1} \times 0$ of $A_n$.

\subsection{Complex structure}

The element \mdfn{$i_n$} $ =(0,1)$ of $A_n$ 
enjoys many special properties.  One of the primary themes of our long-term
project is to fully exploit these special properties.

Let \mdfn{$\C_n$} be 
the $\R$-linear span of $1 = (1,0)$ and $i_n$.  It is a subalgebra of $A_n$
that is isomorphic to $\C$.
An easy consequence of \cite[Lem.~5.5]{DDD} is that $a^*(ai_n) = (a^* a)i_n$
for all $a$ in $A_n$.

\begin{lemma}[DDD, Prop.~5.3]
\label{lem:C-vs}
Under left multiplication, $A_n$ is a $\C_n$-vector space.  In particular,
if $\alpha$ and $\beta$ belong to $\C_n$ and $x$ belongs to $A_n$, then
$\alpha(\beta x) = (\alpha \beta) x$.
\end{lemma}

As a consequence, the expression $\alpha \beta x$ is unambiguous;
we will usually simplify notation in this way.

The \mdfn{real part $\Re(x)$} of an element $x$ of $A_n$ is defined to
be $\frac{1}{2}(x + x^*)$, while the 
\mdfn{imaginary part $\Im(x)$} is defined to be $x - \Re(x)$.

The algebra $A_n$ becomes a positive-definite real inner product space
when we define \mdfn{$\langle a, b \rangle_{\R}$} $= \Re(ab^*)$
\cite[Prop.~3.2]{DDD}.
Also, $A_n$ becomes a positive-definite Hermitian inner product space
when we define \mdfn{$\langle a, b \rangle_{\C}$}
to be the orthogonal projection of $ab^*$ onto the subspace $\C_n$ of $A_n$
\cite[Prop.~6.3]{DDD}.
We say that two elements $a$ and $b$ are \mdfn{$\C$-orthogonal}
if $\langle a, b \rangle_{\C} = 0$.

For any $a$ in $A_n$, let \mdfn{$L_a$} and \mdfn{$R_a$} be
the linear maps $A_n \map A_n$ given by left and right multiplication
by $a$ respectively.

\begin{lemma}[M1, Lem.~1.3, DDD, Lem.~3.4]
\label{lem:adjoint}
Let $a$ be any element of $A_n$.  With respect to the real inner product
on $A_n$, 
the adjoint of $L_a$ is $L_{a^*}$, and the
adjoint of $R_a$ is $R_{a^*}$.
\end{lemma}

We will need the following slightly technical result.

\begin{lemma}
\label{lem:ortho1}
Let $x$ and $y$ be elements of $A_n$ such that $y$ is imaginary.  Then 
$x$ and $xy$ are orthogonal.
\end{lemma}

\begin{proof}
We wish to show that $\langle x, xy \rangle_{\R}$ is zero.
By Lemma \ref{lem:adjoint}, this equals
$\langle x^* x, y \rangle_{\R}$, which is zero because
$x^* x$ is real while $y$ is imaginary.
\end{proof}

We will frequently consider the subspace
\mdfn{$\C_n^\perp$} of $A_n$; it is the orthogonal complement of $\C_n$
(with respect either to the real or to the Hermitian inner product).
Note that $\C_n^\perp$ is a $\C_n$-vector space; in other words,
if $a$ belongs to $\C_n^\perp$ and $\alpha$ belongs to $\C_n$,
then $\alpha a$ also belongs to $\C_n^\perp$ \cite[Lem.~3.8]{DDD}.

\begin{lemma}[DDD, Lem.~6.4 and 6.5]
\label{lem:C-conj-linear}
If $a$ belongs to $\C_n^\perp$, then $L_a$ is $\C_n$-conjugate-linear
in the sense that $L_a(\alpha x) = \alpha^* L_a(x)$ for any 
$x$ in $A_n$ and any $\alpha$ in $\C_n$.
Moreover, $L_a$ is anti-Hermitian in the sense that
$\langle L_a x, y \rangle_{\C} = -\langle x, L_a y \rangle_{\C}^*$.
\end{lemma}

Similar results hold for $R_a$.  See also \cite[Lem.~2.3]{MG} for
a different version of the claim about conjugate-linearity.

The conjugate-linearity of $L_a$ is fundamental to many later calculations.
To emphasize this point, we provide a few exercises.

\begin{exercise}
\label{ex:multiply}
Suppose that $a$ and $b$ belong to $\C_n^\perp$, while $\alpha$
belongs to $\C_n$.  Show that:
\begin{enumerate}
\item
$\alpha a = a \alpha^*$.  
\item
$a \cdot \alpha b = \alpha^* \cdot ab$.
\item
$\alpha a \cdot b = ab \cdot \alpha$.
\end{enumerate}
\end{exercise}

\begin{exercise}
\label{ex:multiply2}
Let $a$ and $b$ belong to $\C_n^\perp$, and let $\alpha$ and $\beta$
belong to $\C_n$.  Suppose also that $a$ and $b$ are $\C$-orthogonal.
Prove that
\[
\alpha a \cdot \beta b = \alpha^* \beta^* \cdot ab.
\]
In this limited sense, multiplication is bi-conjugate-linear.
\end{exercise}

\subsection{Norms}

Norms of elements in $A_n$ are defined with respect to either the
real or Hermitian inner
product: $\norm{a} = \sqrt{\langle a, a\rangle_{\R}} =
\sqrt{\langle a, a\rangle_{\C}} = \sqrt{ aa^*}$; this makes sense
because $aa^*$ is always a non-negative real number \cite[Lem.~3.6]{DDD}.
Note also that $\norm{a} = \norm{a^*}$ for all $a$.

\begin{lemma}
\label{lem:C-norm}
If $a$ belongs to $\C_n^\perp$ and $\alpha$ belongs to $\C_n$,
then $\norm{\alpha a} = \norm{\alpha} \norm{a}$.
\end{lemma}

\begin{proof}
By Lemmas \ref{lem:C-vs} and \ref{lem:C-conj-linear},
\[
\langle \alpha a, \alpha a \rangle_{\C} = 
\langle \alpha^* \alpha a, a \rangle_{\C} =
\langle \norm{\alpha}^2 a, a \rangle_{\C} =
\norm{\alpha}^2 \norm{a}^2.
\]
\end{proof}

\begin{lemma}
\label{lem:norm-comm}
For any $x$ and $y$ in $A_n$, 
$\norm{xy} = \norm{xy^*}$ and $\norm{xy} = \norm{yx}$.
\end{lemma}

\begin{proof}
Since $y + y^*$ is real and $y - y^*$ is imaginary, 
Lemma \ref{lem:ortho1} implies that $\frac{1}{2}x(y+y^*)$ and 
$\frac{1}{2}x(y-y^*)$ are
orthogonal.  But $xy = \frac{1}{2}x(y+y^*) + \frac{1}{2}x(y-y^*)$ and
$xy^* = \frac{1}{2} x(y+y^*) - \frac{1}{2}x(y-y^*)$, so
\[
\norm{xy}^2 = \norm{\frac{1}{2}x(y+y^*)}^2 + \norm{\frac{1}{2}x(y-y^*)}^2 = 
\norm{xy^*}^2.
\]
This establishes the first part of the lemma.

For the second part, recall that $(xy^*)^* = yx^*$.  Then
\[
\norm{xy} = \norm{xy^*} = \norm{yx^*} = \norm{yx},
\]
where the first and third equalities are the first part of the lemma and the
second equality is the fact that conjugation preserves norms.
\end{proof}

\subsection{Standard basis}
\label{subsctn:alt}

The algebra $A_n$ is equipped with an inductively defined 
\mdfn{standard $\R$-basis} \cite[Defn.~2.10]{DDD}.  
The standard $\R$-basis is
orthonormal.

\begin{defn}
An element $a$ of $A_n$ is \mdfn{alternative} if $a \cdot ax = a^2 x$
for all $x$.  An algebra is said to be alternative if all of its
elements are alternative.
\end{defn}

The Cayley-Dickson algebra $A_n$ is alternative if and only if $n \leq 3$.

\begin{lemma}[DDD, Lem.~4.4]
\label{lem:basis-alt}
Standard basis elements are alternative.
\end{lemma}

\subsection{Subalgebras}
\label{subsctn:subalg}

A subalgebra of $A_n$ is an $\R$-linear subspace containing $1$ that is closed
under both multiplication and conjugation.

\begin{defn}
\label{defn:HH_a}
For any elements $a_1, a_2, \ldots, a_k$ in $A_n$, 
let \mdfn{$\llangle a_1, a_2, \ldots, a_k \rrangle$}
denote the smallest subalgebra of $A_n$ that contains the
elements $a_1, a_2, \ldots, a_k$.
\end{defn}

We will usually apply this construction to two elements $a$ and $i_n$.
If $a$ does not belong to $\C_n$, then the subalgebra 
$\llangle a, i_n \rrangle$ 
has an additive basis consisting of $1$, $a$, $i_n$, and $i_n a$
and is isomorphic to the quaternions
\cite[Lem.~5.6]{DDD}.

Because of non-associativity, some properties of generators of Cayley-Dickson
algebras are counter-intuitive.
For example, the algebra $A_3$ is generated by three elements but not by
any two elements.  On the other hand, 
$A_4$ is generated by a generic pair of elements.

\subsection{The octonions}
\label{subsctn:oct}

We recall some properties of $A_3$ and establish some notation.

In $A_3$, we write 
\mdfn{$i$} $= i_1$,
\mdfn{$j$} $= i_2$, \mdfn{$k$} $= ij$, and
\mdfn{$t$} $= i_3$ because it makes the notation less cumbersome.
The standard basis for $A_3$ is
\[
\{ 1, i, j, k, t, it, jt, kt \}.
\]

The automorphism group of $A_3$ is the 14-dimensional sporadic 
Lie group $G_2$ \cite[Sec.~4.1]{B} \cite[Sec.~7]{DDD}.  It acts transitively
on the imaginary elements of length 1.  In other words, up to
automorphism, all imaginary unit vectors are the same.
In fact, $\Aut(A_3)$ acts transitively on ordered pairs of
orthogonal imaginary elements of unit length. 
Even better,
$\Aut(A_3)$ acts transitively on ordered triples $(x,y,z)$ of
pairwise orthogonal imaginary elements of unit length such that $z$ is also
orthogonal to $xy$.  

The subalgebra $\C_3$ is additively generated by $1$ and $t$.  However,
up to automorphism, we may assume that $\C_3$ is generated by $1$
together with any non-zero imaginary element.  Similarly, up to 
automorphism, we may assume that any imaginary element of $A_3$
is orthogonal to $\C_3$.  Such assumptions may not be made in $A_n$
for $n \geq 4$ because the automorphism group of $A_n$ does 
not act transitively \cite{Br} \cite{ES}.


\section{Eigentheory}

\begin{defn}
Let $a$ be a non-zero element of $A_n$.  Define \mdfn{$M_a$} to be the
$\R$-linear map $\frac{1}{\norm{a}^2} L_{a^*} L_{a}$.  The \mdfn{eigenvalues
of $a$} are the eigenvalues of $M_a$.  Similarly, the eigenvectors of
$a$ are the eigenvectors of $M_a$.  Let \mdfn{$\Eig_{\lambda}(a)$}
be the $\lambda$-eigenspace of $a$.
\end{defn}

For any real scalar $r$, the eigenvalues and eigenvectors of $ra$
are the same as those of $a$.  Therefore, we will assume that
$\norm{a} = 1$ whenever it makes our results easier to state.

\begin{remark}
If $a$ is imaginary, then $a^* = -a$.  In this case,
$x$ is a $\lambda$-eigenvector of $a$ if and
only if $a \cdot ax = -\lambda\norm{a}^2 x$.  
\end{remark}

\begin{lemma}
\label{lem:M-conj}
For any $a$ in $A_n$, $M_a$ equals $M_{a^*}$.
\end{lemma}

\begin{proof}
Because $\norm{a} = \norm{a^*}$,
the claim is that $a^* \cdot ax = a \cdot a^*x$ 
for all $a$ and $x$ in $A_n$.  
To check this,
write $a=r+a'$ where $r$ is real and $a'$ is imaginary.  Compute directly
that
\[
(r+a') \cdot (r-a')x = (r-a') \cdot (r+a')x
\]
for all $x$ in $A_n$.  
\end{proof}

\begin{remark}
\label{rem:theta}
To make sense of the notation in the following proposition, 
note that any unit vector in $A_n$
can be written in the form $a \cos \theta + \beta \sin \theta$,
where $a$ and $\beta$ are both unit vectors with $a$ in $\C_n^\perp$
and $\beta$ in $\C_n$.  Generically, $a$ and $\beta$ are unique
up to multiplication by $-1$, and $\theta$ is unique up to the obvious
redundancies of trigonometry.
\end{remark}

\begin{lemma}
\label{lem:M-decomp}
Let $a$ be a unit vector in $\C_n^\perp$, and let $\beta$
be a unit vector in $\C_n$.  Then $M_{a \cos \theta + \beta \sin \theta}$
equals $I \sin^2 \theta + M_{a}\cos^2 \theta$.
\end{lemma}

\begin{proof}
First note that the conjugate of $a \cos \theta + \beta \sin \theta$
is $-a \cos \theta + \beta^* \sin \theta$.  Distribute to compute that
\begin{align*}
(-a\cos \theta + \beta^* \sin \theta) \cdot (a\cos \theta + \beta\sin \theta)x
    = & &\\
\beta^*\beta x \sin^2 \theta 
-a \cdot \beta x\cos \theta \sin \theta
+ \beta^* \cdot ax\cos \theta \sin \theta 
-a \cdot ax\cos^2 \theta. & &
\end{align*}
Using that $\beta^* \beta = \norm{\beta}^2$ and 
that $a \cdot \beta x = \beta^*\cdot ax$ by Lemma \ref{lem:C-conj-linear}, 
this simplifies to 
\[
x\sin^2 \theta + a^* \cdot ax\cos^2 \theta.
\]
\end{proof}

\begin{lemma}
\label{lem:eig-i_n}
For any $a$ in $A_n$, the map $M_a$ is $\C_n$-linear.  In particular,
every eigenspace of $a$ is a $\C_n$-vector space.
\end{lemma}

\begin{proof}
We may assume that $a$ is a unit vector.  Lemma \ref{lem:M-decomp}
allows us to assume that $a$ is imaginary.
Then Lemma \ref{lem:C-conj-linear} says that $M_a$ is the composition
of two conjugate-linear maps, which means that it is $\C_n$-linear.
\end{proof}

The next result is a technical lemma that will be used in many of our 
calculations.  

\begin{lemma}
\label{lem:M-L-adj}
If $a$, $x$, and $y$ belong to $A_n$, then
\[
\langle L_a x, L_a y \rangle_{\R} = \norm{a}^2 \langle M_a x, y \rangle_{\R} =
\norm{a}^2 \langle x, M_a y \rangle_{\R}.
\]
\end{lemma}

\begin{proof}
This follows immediately from the adjointness properties of 
Lemma \ref{lem:adjoint}.
\end{proof}

\begin{lemma}
\label{lem:M_a-zd}
For any $a$ in $A_n$, the kernels of $M_a$ and $L_a$ are equal.  
In particular,
$a$ is a zero-divisor if and only if $0$ is an eigenvalue of $a$.
\end{lemma}

\begin{proof}
If $ax = 0$, then $a^* \cdot ax = 0$.

For the other direction, suppose that $a^* \cdot ax = 0$.  
This implies that $\langle M_a x, x \rangle_{\R}$ equals zero, so
Lemma \ref{lem:M-L-adj} implies that $\langle L_a x, L_a x \rangle_{\R}$
equals zero.  In other words, $\norm{ax}^2 = 0$,
so $ax = 0$.
\end{proof}

\begin{prop}
\label{prop:non-neg}
For every $a$ in $A_n$, $M_a$ is diagonalizable with
non-negative eigenvalues.
If $\lambda_1$ and $\lambda_2$ are distinct
eigenvalues of $a$, then $\Eig_{\lambda_1}(a)$ and
$\Eig_{\lambda_2}(a)$ are orthogonal.
\end{prop}

\begin{proof}
Recall from Lemma \ref{lem:adjoint} that
the adjoint of $L_a$ is $L_{a^*}$.  Therefore, $L_{a^*} L_{a}$ is symmetric;
this shows that $M_a$ is also symmetric.
The fundamental theorem of symmetric matrices says that
$M_a$ is diagonalizable.
The orthogonality of the eigenspaces is a standard
property of symmetric matrices.

To show that all of the eigenvalues are non-negative,
let $M_a x=\lambda x$ with $\lambda \neq 0$
and $x\neq 0$.  The value
$\langle x, x \rangle_{\R}$ is positive, and it
equals
\[
\frac{1}{\lambda} \langle \lambda x, x \rangle_{\R} =
\frac{1}{\lambda} \langle  M_a x, x \rangle_{\R} =
\frac{1}{\lambda} \langle  L_{a} x, L_{a} x \rangle_{\R}
\]
by Lemma \ref{lem:M-L-adj}.  Since $\langle L_a x, L_a x \rangle_{\R}$ is
also positive,
it follows that $\lambda$ must be positive.
\end{proof}

In practice, we will only study eigenvalues of elements of $A_n$
that are orthogonal to $\C_n$.
The result below explains that if we understand the
eigenvalues in this special case, then we understand them all.

Recall from Remark \ref{rem:theta} that any unit vector in $A_n$
can be written in the form $a \cos \theta + \beta \sin \theta$,
where $a$ is a unit vector in $\C_n^\perp$ and $\beta$ is a unit vector
in $\C_n$.

\begin{prop}
\label{prop:reduce-C-perp}
Let $a$ and $\beta$ be unit vectors in $A_n$ such that
$a$ belongs to $\C_n^\perp$ and $\beta$ belongs to $\C_n$.
Then
\[ \Eig_\lambda(a)=\Eig_{\sin^2 \theta +\lambda \cos^2\theta}
                       (a \cos \theta + \beta \sin \theta).
\]
In particular, $\lambda$ is an
eigenvalue of $a$ if and only if 
$\sin^2 \theta+\lambda \cos^2\theta$ is an
eigenvalue of $a\cos \theta + \beta \sin \theta$.  
\end{prop}

\begin{proof}
This follows immediately from Lemma \ref{lem:M-decomp}, which says
that $M_{a\cos \theta + \beta \sin \theta}$
equals $I \sin^2 \theta + M_{a}\cos^2 \theta$.
\end{proof}

\begin{remark}
\label{rem:reduce-C-perp}
Note that the case $\lambda=1$ is special in the above proposition,
giving that $\Eig_1(a)=\Eig_1(a\cos\theta+\beta\sin\theta)$.  
In other words, the $1$-eigenspace of an element of $A_n$
depends only on its orthogonal projection onto $\C_n^\perp$.
\end{remark}

\begin{remark}
\label{rem:zd}
Let $a$ and $\beta$ be unit vectors in $A_n$ such that
$a$ belongs to $\C_n^\perp$ and $\beta$ belongs to $\C_n$.
Propositions \ref{prop:non-neg} and \ref{prop:reduce-C-perp} show that 
the eigenvalues of $a\cos \theta + \beta\sin \theta$ are at least
$\sin^2 \theta$.
In particular, if $0$ is an eigenvalue of 
$a \cos \theta + \beta \sin \theta$, then $\sin \theta = 0$.
In other words, zero-divisors are always orthogonal to $\C_n$ 
\cite[Cor.~1.9]{M1} \cite[Lem.~9.5]{DDD}.
\end{remark}

Recall from Section \ref{subsctn:subalg} that $\llangle a,i_n\rrangle$ is the
subalgebra generated by $a$ and $i_n$.

\begin{prop}
\label{prop:eig-H_a}
For any $a$ in $A_n$, $\llangle a,i_n \rrangle$ is contained in $\Eig_1(a)$.
In particular, $1$ is an eigenvalue of every non-zero element of $A_n$.
\end{prop}

\begin{proof}
First note that $\llangle a,i_n \rrangle$ is isomorphic to 
either $\C$ or $\HH$; this follows from \cite[Lem.~5.6]{DDD}.
In either case, it is an associative subalgebra.
Therefore, $a^* \cdot a x = a^*a \cdot x = \norm{a}^2 x$ for any 
$x$ in $\llangle a, i_n \rrangle$.
\end{proof}

\begin{lemma}
\label{lem:M-beta-a}
For any $a$ in $\C_n^\perp$ and any $\beta$ in $\C_n$,
$M_a = M_{\beta a}$.
\end{lemma}

\begin{proof}
We may assume that $a$ and $\beta$ both have norm $1$.

First note that $\llangle a,i_n \rrangle$ equals 
$\llangle \beta a, i_n \rrangle$.  
By Proposition \ref{prop:eig-H_a}, $M_a$ and $M_{\beta a}$ are equal
on this 4-dimensional subspace.

Because $M_a$ and $M_{\beta a}$ are both $\C_n$-linear 
by Lemma \ref{lem:eig-i_n},
we only need to verify that $M_a(x) = M_{\beta a} (x)$ for
$x$ in $\C_n^\perp$ such that $a$ and $x$ are $\C$-orthogonal.
Compute
\[ 
(\beta a)^* \cdot (\beta a)x = \beta^* \beta (a^* \cdot ax) =
\norm{\beta}^2 a^* \cdot ax = a^* \cdot ax
\]
using Lemma \ref{lem:C-conj-linear}.  In this computation, 
we need that $ax$ is orthogonal to $\C_n$;
this is equivalent to the assumption that $a$ and $x$ are $\C$-orthogonal.
\end{proof}

\begin{prop}
\label{prop:eig-C_n}
Let $a$ and $\beta$ be non-zero vectors in $A_n$
such that $\beta$ belongs to $\C_n$.
Then $\Eig_\lambda(a)=\Eig_{\lambda}(\beta a)$ for any $\lambda$.
In particular, the eigenvalues of $a$ and $\beta a$ are the same.
\end{prop}

See also \cite[Cor.~3.6]{MG} for a related result in different
notation.

\begin{proof}
We may assume that $a$ and $\beta$ both have norm $1$.
Proposition \ref{prop:reduce-C-perp}
implies that the result holds for all $a$ if it holds
for $a$ in $\C_n^\perp$.  Therefore, we may assume that
$a$ is orthogonal to $\C_n$.
Then Lemma \ref{lem:M-beta-a} gives the desired result immediately.
\end{proof}

\begin{lemma}
\label{lem:bilin-form}
For all $x$ and $y$ in $A_n$, 
$\tr (L_{x^*} L_y)$ equals $2^n \langle x, y\rangle_{\R}$.
\end{lemma}

\begin{proof}
Recall the standard basis described in Section \ref{subsctn:alt}.
We want to compute
\[
\sum_z \langle z, L_{x^*} L_y z \rangle_{\R},
\]
where $z$ ranges over the standard basis.  Using the adjointness of
Lemma \ref{lem:adjoint}, compute that
\[
\sum_z \langle z, L_{x^*} L_y z \rangle_{\R} =
\sum_z \langle xz, y z \rangle_{\R} =
\sum_z \langle xz \cdot z^*, y \rangle_{\R} =
\sum_z \langle x, y \rangle_{\R} = 
2^n \langle x, y \rangle_{\R},
\]
where the third equality uses that $z$ is alternative 
by Lemma \ref{lem:basis-alt}.
\end{proof}

\begin{prop}
\label{prop:eig-sum}
For any $a$ in $A_n$,
the sum of the eigenvalues of $a$ is equal to $2^n$.
\end{prop}

\begin{proof}
This follows immediately from Lemma \ref{lem:bilin-form}
because the trace of a diagonalizable operator equals the sum of its
eigenvalues.
\end{proof}

\begin{lemma}
\label{lem:L-restrict}
For any $a$ in $A_n$, the map $L_a$ takes $\Eig_\lambda(a)$ into
$\Eig_\lambda(a)$.  If $\lambda$ is non-zero, then $L_a$ restricts to
an automorphism of $\Eig_\lambda(a)$.
\end{lemma}

\begin{proof}
Let $x$ belong to $\Eig_\lambda(a)$.  
Using that $M_a = M_{a^*}$ from Lemma \ref{lem:M-conj},
compute that
\[
M_a(ax) = M_{a^*}(ax) =
\frac{1}{\norm{a}^2}a \cdot a^* (ax) = 
a \cdot M_a x = \lambda ax.
\]
This shows that $ax$ also belongs to $\Eig_\lambda(a)$.

For the second claim, simply note that $L_{a^*} L_{a}$ is
scalar multiplication by $\lambda \norm{a}^2$ on $\Eig_\lambda(a)$.
Thus the inverse to $L_a$ is $\frac{1}{\lambda\norm{a}^2}L_{a^*}$.
\end{proof}

\begin{remark}
\label{rem:L-restrict}
For $\lambda \neq 0$, the restriction 
$L_a :\Eig_\lambda(a) \map \Eig_\lambda(a)$ is a similarity
in the sense that it is an isometry up to scaling.  This follows
from Lemma \ref{lem:M-L-adj}.

Also, an immediate consequence of Lemma \ref{lem:L-restrict} is
that the image of $L_a$ is equal to the orthogonal complement
of $\Eig_0(a)$.  This fact is used in \cite{DDDD}.
\end{remark}

\begin{prop}
\label{prop:eig-mod-4}
Let $n \geq 2$.
For any $a$ in $A_n$, 
every eigenspace of $a$ is even-dimensional over $\C_n$.
In particular, the real dimension of any eigenspace is a multiple of 4.
\end{prop}

See also \cite[Thm.~4.6]{MG} for the second claim.

\begin{proof}
By Proposition \ref{prop:reduce-C-perp}, we may assume that 
$a$ is orthogonal to $\C_n$.  Let $\lambda$ be an eigenvalue
of $a$.

Recall from Lemma \ref{lem:C-conj-linear}
that $L_a$ is a conjugate-linear anti-Hermitian map.
If $\lambda$ is non-zero, then Lemma \ref{lem:L-restrict}
says that $L_a$ restricts to an automorphism of $\Eig_\lambda(a)$.
By \cite[Lem.~6.6]{DDD}, conjugate-linear anti-Hermitian 
automorphisms exist only on even-dimensional $\C$-vector spaces.

Now consider $\lambda = 0$.  The $\C_n$-dimension of $\Eig_0(a)$ is equal
to $2^{n-1}$ minus the dimensions of the other eigenspaces.
By the previous paragraph and the fact that $2^{n-1}$ is even,
it follows that $\Eig_0(a)$ is also even-dimensional.
\end{proof}

The previous proposition, together with 
Propositions \ref{prop:non-neg} and \ref{prop:eig-sum},
shows that if $a$ belongs
to $A_n$, then the eigenvalues of $a$ are at most $2^{n-2}$.  However,
this bound is not sharp.  Later in Corollary \ref{cor:eig-bound}
we will prove a stronger result.

\begin{prop}
\label{prop:eig-norm}
Let $a$ belong to $A_n$, and let $\lambda\geq 0$.  
Then $x$ belongs to $\Eig_\lambda(a)$
if and only if $\norm{ax}=\sqrt{\lambda}\norm{a} \norm{x}$ and
$\norm{M_a x} = \lambda \norm{x}$.
\end{prop}

The above proposition is a surprisingly strong result.  We know that 
$L_{a^*}L_a$
scales an element of $\Eig_{\lambda}(a)$ by $\lambda\norm{a}^2$.  
The proposition
makes the non-obvious claim 
that this scaling occurs in two geometrically equal stages for the
two maps $L_{a^*}$ and $L_a$.
Moreover, it says that as long
as the norms of $L_a x$ and $M_a x$ are correct, then the direction
of $M_a x$ takes care of itself.  In practice, it is a very 
useful simplification not to have
to worry about the direction of $M_a x$.
One part of Proposition \ref{prop:eig-norm} is proved
in \cite[Prop.~4.20]{MG}.

\begin{proof}
First suppose that $x$ belongs to $\Eig_\lambda(a)$.
The second desired equality follows immediately.
For the first equality, use 
Lemma \ref{lem:M-L-adj} to compute that 
\[
\norm{ax}^2 = \norm{a}^2 \langle M_a x, x \rangle_{\R} =
\lambda \norm{a}^2 \norm{x}^2.
\]
Now take square roots.
This finishes one direction.

For the other direction, note that $x$ belongs to $\Eig_\lambda(a)$
if and only if the norm of $M_a x - \lambda x$ is zero.
Using the formulas in the proposition and Lemma \ref{lem:M-L-adj},
compute that
\begin{align*}
\langle M_a x - \lambda x, M_a x - \lambda x \rangle_{\R} & = & \\
\norm{M_a x}^2 + \lambda^2 \norm{x}^2 - 
2\lambda \langle M_a x, x \rangle_{\R} & = & \\
\lambda^2 \norm{x}^2 + \lambda^2 \norm{x}^2 - 
2\lambda \frac{\norm{ax}^2}{\norm{a}^2} &=&\\
2 \lambda^2 \norm{x}^2 - 
2\lambda \frac{\lambda \norm{a}^2 \norm{x}^2}{\norm{a}^2} &=0.&
\end{align*}
\end{proof}


\section{Maximum and minimum eigenvalues}

\begin{defn}
For $a$ in $A_n$, let \mdfn{$\lambda^-_a$} denote the minimum eigenvalue of
$a$, and let \mdfn{$\lambda^+_a$} denote the maximum eigenvalue of $a$.
\end{defn}

Recall from Proposition \ref{prop:eig-H_a} that $1$ is always an eigenvalue
of $a$ if $a$ is non-zero.  Therefore, $\lambda^+_a$ is always at least 1,
$\lambda^-_a$ is at most 1, and $\lambda^-_a = \lambda^+_a$ if and only if
$a$ is alternative.

\begin{defn}
\label{defn:eigen-decomp}
Let $a$ and $x$ belong to $A_n$.  The \mdfn{eigendecomposition of $x$
with respect to $a$} is the sum $x = x_1 + \cdots + x_k$ where
each $x_i$ is an eigenvector of $a$ with eigenvalue $\lambda_i$ such
that the $\lambda_i$ are distinct.
\end{defn}

Note that the eigendecomposition of $x$ with respect to $a$ is unique
up to reordering.  Note also from Proposition \ref{prop:non-neg}
that the eigendecomposition
of $x$ is an orthogonal decomposition in the sense that
$x_i$ and $x_j$ are orthogonal for distinct $i$ and $j$.

\begin{lemma}
\label{lem:eigen-decomp}
Let $a$ and $x$ belong to $A_n$, and let $x_1 + \cdots + x_k$ be
the eigendecomposition of $x$ with respect to $a$.  Then
$ax_1 + \cdots + ax_k$ is the eigendecomposition of $ax$ with 
respect to $a$, except that the term $ax_i$ must be removed if
$x_i$ belongs to $\Eig_0(a)$.
\end{lemma}

\begin{proof}
This follows immediately from Lemma \ref{lem:L-restrict}.
\end{proof}

\begin{prop}
\label{prop:eig-boundnorm}
Let $a$ belong to $A_n$.  For any $x$ in $A_n$,
\begin{enumerate}
\item
$\sqrt{\lambda^-_a} \norm{a} \norm{x} \leq \norm{ax} \leq 
 \sqrt{\lambda^+_a}\norm{a} \norm{x}$.
\item
$\norm{ax}=\sqrt{\lambda^+_a} \norm{a} \norm{x}$ if and only if
$x$ belongs to $\Eig_{\lambda^+_a}(a)$.
\item
$\norm{ax}=\sqrt{\lambda^-_a} \norm{a} \norm{x}$ if and only if
$x$ belongs to $\Eig_{\lambda^-_a}(a)$.
\end{enumerate}
\end{prop}

\begin{proof}
Let $x=x_1+\cdots+x_k$ be the eigendecomposition of $x$ with respect to $a$,
so $ax=ax_1+\cdots+ax_k$ is the eigendecomposition of $ax$ with respect to $a$
by Lemma \ref{lem:eigen-decomp}
(except possibly that one term must be dropped).
Let $\lambda_i$ be the eigenvalue of $x_i$ with respect to $a$.
Using Proposition \ref{prop:eig-norm} and using that
eigendecompositions are orthogonal decompositions by
Proposition \ref{prop:non-neg},
we have
\begin{align*}
 \norm{ax}^2=\norm{ax_1}^2 + \cdots + \norm{ax_k}^2  
&= \lambda_1 \norm{a}^2\norm{x_1}^2 + \cdots + 
\lambda_k \norm{a}^2\norm{x_k}^2 \\
&\leq \lambda^+_a \norm{a}^2 
\left( \norm{x_1}^2+\cdots + \norm{x_k}^2 \right) \\
&= \lambda^+_a  \norm{a}^2 \norm{x}^2,
\end{align*}
where equality holds if and only if $x$ belongs to $\Eig_{\lambda_a^+}(a)$.
This proves half of part (1) and also part (2).

The remaining parts of the lemma, involving $\lambda^-_a$,
are derived similarly. 
\end{proof}

\begin{lemma} 
\label{lem:eig-boundnorm}
Let $a$ belong to $A_n$.  For any $x$ in $A_n$,
\[
\sqrt{\lambda^-_a} \norm{a} \norm{x} \leq \norm{ax} \leq 
 \sqrt{\lambda^+_a}\norm{a} \norm{x}.
\]
\end{lemma}

\begin{proof}
\end{proof}

\begin{prop}
\label{prop:eig-maxboundnorm}
Let $a$ belong to $A_n$.
Then $x$ belongs to $\Eig_{\lambda^+_a}(a)$ if and only if
$\norm{ax}=\sqrt{\lambda^+_a} \norm{a} \norm{x}$.
Also, $x$ belongs to $\Eig_{\lambda^-_a}(a)$ if and only if
$\norm{ax}=\sqrt{\lambda^-_a} \norm{a} \norm{x}$.
\end{prop}

The reader should compare this result to Proposition \ref{prop:eig-norm}.
We are claiming that 
for the minimum and maximum eigenvalues, the second condition is redundant.

\begin{proof}
We give the proof of the first statement; the proof of the second
statement is the same.

One direction is an immediate consequence of Proposition \ref{prop:eig-norm}.
For the other direction,
suppose that $\norm{ax}=\sqrt{\lambda^+_a} \norm{a} \norm{x}$.   
Let $x=x_1+\cdots+x_k$ be the eigendecomposition of $x$ with
respect to $a$, so 
$ax=ax_1+\cdots+ax_k$ is the eigendecomposition of $ax$ with
respect to $a$ (except possibly that one term must be dropped).
Let $\lambda_i$ be the eigenvalue of $x_i$ with respect to $a$.

By Proposition \ref{prop:eig-norm}, we have
\begin{align*}
\lambda^+_a \norm{a}^2 \left( \norm{x_1}^2+\cdots+\norm{x_k}^2 \right) =
\lambda^+_a \norm{a}^2\norm{x}^2 
&=\norm{ax}^2\\
&=\norm{ax_1}^2+\cdots+\norm{ax_k}^2 \\
&=\norm{a}^2 \left( \lambda_1 \norm{x_1}^2+\cdots + 
\lambda_k \norm{x_k}^2 \right).
\end{align*}
Rearrange this equality to get
\[ (\lambda^+_a-\lambda_1)\norm{x_1}^2+\cdots + 
   (\lambda^+_a-\lambda_k)\norm{x_k}^2=0.
\]
The coefficients $\lambda^+_a-\lambda_i$ are all positive except for the
one value of $j$ for which $\lambda_j=\lambda^+_a$.  
It follows that $x_i = 0$ for $i \neq j$ and hence
$x=x_j$, so $x$ belongs to $\Eig_{\lambda^+_a}(a)$.  
\end{proof}

\begin{prop}
\label{prop:eig-double}
Let $a = (b,c)$ belong to $A_n$, where $b$ and $c$ are elements of $A_{n-1}$.
Then $\lambda^+_{a} \leq 2\max\{\lambda^+_b,\lambda^+_c\}$.
\end{prop}

\begin{proof}
Let $x = (y,z)$ belong to $A_n$, and
consider the product
\[ 
ax = (b,c)(y,z) = (by-z^*c,cy^*+zb).
\]
The triangle inequality and Lemma \ref{lem:norm-comm} says that
\begin{align*}
\norm{ax}
& = \sqrt{ \norm{by-z^*c}^2 + \norm{cy^* + zb}^2 } \\
& \leq \sqrt{ \norm{by}^2 + \norm{cz^*}^2 + \norm{cy^*}^2 + \norm{bz}^2 
+ 2 \norm{by}\norm{cz^*} + 2 \norm{cy^*} \norm{bz} }.
\end{align*}
Writing $B = \norm{b}$, $C = \norm{c}$, $Y = \norm{y}$, and $Z = \norm{z}$,
repeated use of Lemma \ref{lem:eig-boundnorm} 
(and the facts that $\norm{z^*} = \norm{z}$ and $\norm{y^*} = \norm{y}$)
gives the inequality
\[
\norm{ax} \leq
\sqrt{ \lambda^+_b B^2 Y^2 + \lambda^+_c C^2 Z^2 +
\lambda^+_c C^2 Y^2 + \lambda^+_b B^2 Z^2
+ 4\sqrt{\lambda^+_b \lambda^+_c} BCYZ}.
\]
Replacing $\lambda^+_b$, $\lambda^+_c$, and $\sqrt{\lambda^+_b \lambda^+_c}$
with $\max\{\lambda^+_b, \lambda^+_c \}$, we obtain
\[
\norm{ax} \leq \sqrt{\max\{\lambda^+_b, \lambda^+_c \} }
\sqrt{ B^2 Y^2 + C^2 Z^2 + C^2 Y^2 + B^2 Z^2
+ 4 BCYZ}.
\]
Next, use the inequalities $2BC \leq B^2 + C^2$ and
$2YZ \leq Y^2 + Z^2$ to get
\[
\norm{ax} \leq 
\sqrt{\max\{\lambda^+_b, \lambda^+_c \} }
\sqrt{ 2 (B^2 + C^2) (Y^2 + Z^2)} = 
\sqrt{2\max\{\lambda^+_b, \lambda^+_c \} }\norm{a}\norm{x}.
\]
Since this inequality holds for all $x$, we conclude by
Proposition \ref{prop:eig-maxboundnorm} that 
$\lambda^+_a \leq 2\max\{\lambda^+_b, \lambda^+_c \}$.
\end{proof}

\begin{cor}
\label{cor:eig-bound}
Let $n \geq 3$.
If $a$ belongs to $A_n$, then all eigenvalues of $a$ are in the 
interval $[0,2^{n-3}]$.
\end{cor}

\begin{proof}
The proof is by induction,
using Propositions \ref{prop:non-neg} and \ref{prop:eig-double}.
The base case is $n = 3$.
Recall that $A_3$ is alternative, so 1 is the only eigenvalue of 
any $a$ in $A_3$.  
\end{proof}

\begin{remark}
Corollary \ref{cor:eig-bound} is sharp in the following sense.
For $n \geq 4$, every real number in the interval $[0, 2^{n-3}]$
occurs as the eigenvalue of some element of $A_n$.  
See Theorem \ref{thm:eig-exist} for more details.
\end{remark}


\section{Cross-product}

\begin{defn}
\label{defn:cross}
Given $a$ and $b$ in $A_n$, let the \mdfn{cross-product $a \times b$}
be the imaginary part of $ab^*$.
\end{defn}

If $\R^3$ is identified with the imaginary part of $A_2$,
then this definition restricts to the usual notion of cross-product 
in physics.  The cross-product has also been previously studied
for $A_3$; see \cite[Sec.~4.1]{B} for example.
We shall see that cross-products are indispensible in describing 
eigenvalues and eigenvectors, especially for $A_4$.

\begin{lemma}
\label{lem:cross}
Let $a$ and $b$ belong to $A_n$, and let $\theta$ be the angle
between $a$ and $b$.  If $b$ belongs to $\Eig_1(a)$ or $a$ 
belongs to $\Eig_1(b)$, then 
\[
\norm{a \times b} = \norm{a} \norm{b} \sin \theta.
\]
\end{lemma}

\begin{proof}
Note that $ab^* = \Re(ab^*) + \Im(ab^*)$ is an orthogonal
decomposition of $ab^*$.  Therefore,
\[
\norm{\Im(ab^*)}^2 = \norm{ab^*}^2 - \norm{\Re(ab^*)}^2 =
\norm{ab^*}^2 - \langle a, b \rangle_{\R}^2.
\]
By Proposition \ref{prop:eig-norm}, we know that
$\norm{ab^*}^2 = \norm{a}^2 \norm{b^*}^2 = \norm{a}^2 \norm{b}^2$, 
so we get that
\[
\norm{\Im(ab^*)}^2 = \norm{a}^2 \norm{b}^2 - 
\norm{a}^2 \norm{b}^2 \cos^2 \theta =
\norm{a}^2 \norm{b}^2 \sin^2 \theta.
\]
\end{proof}

\begin{lemma}
\label{lem:cross2}
Let $a$ and $b$ belong to $A_n$
such that $b$ belongs to $\Eig_1(a)$ or $a$ 
belongs to $\Eig_1(b)$.  Then
$\norm{a \times b} \leq \frac{1}{2}( \norm{a}^2 + \norm{b}^2)$.
Moreover, $\norm{a \times b} = \frac{1}{2}(\norm{a}^2 + \norm{b}^2)$ 
if and only if $a$
and $b$ are orthogonal and have the same norm.
Also $a \times b = 0$ if and only if $a$ and $b$ are linearly dependent.
\end{lemma}

\begin{proof}
The inequality follows from Lemma \ref{lem:cross} together with the
simple observation that
\[
\norm{a}^2 + \norm{b}^2 \geq 2\norm{a} \norm{b} \geq 
2\norm{a} \norm{b} \sin \theta.
\]
It then follows that $\norm{a \times b} = \frac{1}{2}(\norm{a}^2 + \norm{b}^2)$
if and only if
$\norm{a}^2 + \norm{b}^2 = 2\norm{a}\norm{b}$ and $\sin \theta = 1$.  These
two conditions occur if and only if $\norm{a} = \norm{b}$ and 
$\theta = \frac{\pi}{2}$.

Finally, Lemma \ref{lem:cross} shows that $a \times b = 0$ if and only if
$a = 0$, $b = 0$, $\theta = 0$, or $\theta = \pi$.  
\end{proof}

\begin{lemma}
\label{lem:cross3}
Let $a$ belong to $\C_{n-1}^\perp$, and let $\alpha$ and $\beta$
belong to $\C_{n-1}$.  Then 
\[
\alpha a \times \beta a = \norm{a}^2  (\alpha \times \beta).
\]
\end{lemma}

\begin{proof}
Use Lemma \ref{lem:C-conj-linear} to 
compute that 
$(\alpha a)(\beta a)^* = \norm{a}^2 \alpha \beta^*$.
\end{proof}

\begin{lemma}
\label{lem:cross-product-ortho}
Let $a$ and $b$ be imaginary elements of $A_n$.  Then
$a \times b$ is orthogonal to both $a$ and $b$.
\end{lemma}

\begin{proof}
Lemma \ref{lem:ortho1} says that $ab$ is orthogonal to both $a$ and $b$.
Also, $a$ and $b$ are orthogonal to $\Re(ab)$ because they are 
imaginary.  Therefore, $a$ and $b$ are orthogonal to $\Im(ab) = ab - \Re(ab)$.
Finally, observe that $\Im(ab^*) = -\Im(ab)$ because $b$ is imaginary.
\end{proof}

In $\HH$ and $\OO$, cross products are useful for producing unit vectors
that are orthogonal to two given vectors.  
Unfortunately, cross products
are not as useful in the higher Cayley-Dickson algebras.
Even though $a \times b$ is always orthogonal to $a$ and $b$ by the previous
lemma, beware that $a \times b$ may equal zero.


\section{Eigenvalues and basic constructions}

Throughout this section, the reader should keep the following
ideas in mind.
We will consider elements of $A_n$ of the form
$(\alpha a, \beta a)$, where $a$ belongs to $\C_{n-1}^\perp$
and $\alpha$ and $\beta$ belong to $\C_{n-1}$.  Under these
circumstances, Lemma \ref{lem:cross3} applies, and we conclude
that $\alpha a \times \beta a$ always belongs to $\C_{n-1}$.  
Moreover, since $\alpha a \times \beta a$
is imaginary, it is in fact an $\R$-multiple of $i_{n-1}$.  
Even more precisely, $\alpha a \times \beta a$
equals $\pm \norm{a}^2 \norm{\alpha \times \beta} i_{n-1}$.

For $a$ in $\C_{n-1}^\perp$, recall from Section \ref{subsctn:subalg}
that $\llangle a,i_{n-1}\rrangle$ is the subalgebra of $A_{n-1}$ 
generated by $a$ and $i_{n-1}$.  It is isomorphic to $\HH$.

\begin{lemma}
\label{lem:mult-alpha}
Let $a$ belong to $\C_{n-1}^\perp$, and let $\alpha$ and $\beta$
belong to $\C_{n-1}$.  If $x$ and $y$ are orthogonal to
$\llangle a,i_{n-1}\rrangle$, then
$(\alpha a, \beta a) \cdot (\alpha a, \beta a) (x,y)$ equals
\[
\Big( (\norm{\alpha}^2 + \norm{\beta}^2) a \cdot ax + 
   2 (\alpha \times \beta) (a \cdot ay),
(\norm{\alpha}^2 + \norm{\beta}^2) a \cdot ay - 
   2 (\alpha \times \beta) (a \cdot ax) \Big). 
\]
\end{lemma}

\begin{proof}
Compute using Lemma \ref{lem:C-conj-linear}.  This is a generalized
version of the computations in \cite[Sec.~10]{DDD}.
\end{proof}

\begin{prop}
\label{prop:eig-alpha}
Let $a$ belong to $\C_{n-1}^\perp$, and let $\alpha$ and $\beta$
belong to $\C_{n-1}$ such that $\norm{a}$ and 
$\norm{\alpha}^2 + \norm{\beta}^2$
both equal $1$ (so that $(\alpha a, \beta a)$ is a unit vector).
Suppose that $\alpha \times \beta$ is non-zero (i.e., $\alpha$ and $\beta$
are $\R$-linearly independent). Let 
$\gamma = \alpha \times \beta/\norm{\alpha \times \beta}$.
\begin{enumerate}[(a)]
\item 
$\llangle a, i_{n-1}, i_n \rrangle$ is contained in the $1$-eigenspace of 
$(\alpha a, \beta a)$;
\item
$\{(x,-\gamma x) : x\in \Eig_1(a)\cap \llangle a, i_{n-1}\rrangle^\perp\}$ 
is contained in
the $(1 + 2 \norm{\alpha \times \beta})$-eigenspace of $(\alpha a, \beta a)$;
\item
$\{(x, \gamma x) : x\in \Eig_1(a)\cap \llangle a,i_{n-1}\rrangle^\perp\}$ 
is contained in
the $(1 - 2 \norm{\alpha \times \beta})$-eigenspace of $(\alpha a, \beta a)$;
\item
$\{(x, -\gamma x): x\in \Eig_{\lambda}(a) \}$ is contained in
the $(1 + 2 \norm{\alpha \times \beta})\lambda$-eigenspace of 
$(\alpha a, \beta a)$;
\item
$\{(x, \gamma x): x\in \Eig_{\lambda}(a) \}$ is contained in
the $(1 - 2 \norm{\alpha \times \beta})\lambda$-eigenspace of 
$(\alpha a, \beta a)$;
\end{enumerate}
\end{prop}

\begin{proof}
Note that $\llangle a, i_{n-1}, i_n\rrangle$ is an algebra 
that contains $(\alpha a, \beta a)$ and is isomorphic to the 
octonions. 
This establishes part (a) because the octonions are alternative.  

Now suppose that $x$ is a $\lambda$-eigenvector of $a$ and
is orthogonal to $\llangle a, i_{n-1}\rrangle$.
By Lemma \ref{lem:eig-i_n}, 
$\pm \gamma x$ is also a $\lambda$-eigenvector of $a$ and is
orthogonal to $\llangle a, i_{n-1}\rrangle$.
Hence Lemma \ref{lem:mult-alpha} applies, and we compute that
$(\alpha a,\beta a) \cdot (\alpha a,\beta a)(x,\pm \gamma x)$ equals
\[
\Big( -\lambda x \mp 2 (\alpha \times \beta) \lambda \gamma x,
\mp \lambda \gamma x + 2 (\alpha \times \beta) \lambda x \Big). 
\]
Recall that $\gamma$ is an imaginary unit vector, so $\gamma^2 = - 1$.
It follows that the above expression equals
\[
-\lambda \left( 1 \mp 2\norm{\alpha \times \beta} \right) 
( x, \pm \gamma x ).
\]
Parts (b) through (e) are direct consequences of this formula.
\end{proof}

\begin{remark}
By counting dimensions, it is straightforward to check that $A_n$
is the direct sum of the subspaces listed in the proposition.  Thus,
the proposition completely describes the eigentheory of 
$(\alpha a, \beta a)$.  
Note also that $\llangle a, i_{n-1}, i_n \rrangle$ consists
of elements of the form $(x,y)$, where $x$ and $y$ both belong to
$\llangle a, i_{n-1} \rrangle$.
Finally, it is important to keep in mind that $\gamma$ always equals
$i_{n-1}$ or $-i_{n-1}$.
\end{remark}

\begin{cor}
\label{cor:gamma-not-0}
Let $a$ belong to $\C_{n-1}^\perp$, and let $\alpha$ and $\beta$
belong to $\C_{n-1}$.
Suppose that $\alpha \times \beta$ is non-zero (i.e., $\alpha$ and $\beta$
are $\R$-linearly independent). 
Every eigenvalue of $(\alpha a, \beta a)$ either equals $1$ or is
of the form
\[
\left( 1 \pm \frac{2\norm{\alpha \times \beta}}
                  {\norm{\alpha}^2 + \norm{\beta}^2} \right) \lambda,
\]
where $\lambda$ is an eigenvalue of $a$.
\end{cor}

\begin{proof}
Let $N = \norm{\alpha}^2 + \norm{\beta}^2$.  Note that
$\norm{(\alpha a, \beta a)}^2$ equals $N\norm{a}^2$ by Lemma \ref{lem:C-norm}.

Now consider the unit vector 
$\left( \frac{\alpha}{\sqrt{N}} \frac{a}{\norm{a}}, 
   \frac{\beta}{\sqrt{N}} \frac{a}{\norm{a}} \right)$.
This vector is an $\R$-multiple of $(\alpha a, \beta a)$, 
so we just need to compute the eigenvalues of this unit vector.

Apply Proposition \ref{prop:eig-alpha} and conclude that the 
eigenvalues of 
$\left( \frac{\alpha}{\sqrt{N}} \frac{a}{\norm{a}}, 
   \frac{\beta}{\sqrt{N}} \frac{a}{\norm{a}} \right)$
are either $1$ or of the form 
$\left(1 \pm 2 \left| \frac{\alpha}{\sqrt{N}} \times 
  \frac{\beta}{\sqrt{N}} \right| \right)\lambda$,
where $\lambda$ is an eigenvalue of $a$.  This expression equals
$\left(1 \pm \frac{2 \norm{\alpha \times \beta}}{N} \right)\lambda$, 
as desired.
\end{proof}

In practice, the multiplicities of the eigenvalues in
Corollary \ref{cor:gamma-not-0} can be computed by inspection of
Proposition \ref{prop:eig-alpha}.  However, precise results are difficult
to state because of various special cases.  For example,
$1 \pm \frac{2\norm{\alpha \times \beta}}{\norm{\alpha}^2 + \norm{\beta}^2}$ 
are eigenvalues of $(\alpha a, \beta a)$ only
if $\Eig_1(a)$ strictly contains $\llangle a, i_{n-1} \rrangle$.
Also, it is possible that 
\[
\left(1-\frac{2\norm{\alpha \times\beta}}
             {\norm{\alpha}^2 +\norm{\beta}^2}\right)\lambda
= \left(1+\frac{2\norm{\alpha \times\beta}}
             {\norm{\alpha}^2 +\norm{\beta}^2}\right)\mu
\]
for distinct eigenvalues $\lambda$ and $\mu$ of $a$.

Because of part (a) of Proposition \ref{prop:eig-alpha},
the dimension of $\Eig_1(\alpha a, \beta a)$ is always at least 8.

\begin{cor}
\label{cor:gamma-max}
Let $a$ belong to $\C_{n-1}^\perp$, and let $\alpha$ and $\beta$
belong to $\C_{n-1}$ such that $\norm{a}$ and 
$\norm{\alpha}^2 + \norm{\beta}^2$
both equal $1$ (so that $(\alpha a, \beta a)$ is a unit vector).
Suppose that $\norm{\alpha \times \beta} = \frac{1}{2}$ 
(equivalently, by Lemma \ref{lem:cross2}, 
$\alpha$ and $\beta$ are orthogonal and have the same norm).
Every eigenvalue of $(\alpha a, \beta a)$ equals 
$0$ or $1$, or is of the form 
$2\lambda$ where $\lambda$ is an eigenvalue of $a$.
Moreover,
\begin{enumerate}[(a)]
\item
the multiplicity of $0$ is equal to $2^{n-1} - 4 + \dim \Eig_0(a)$;
\item
the multiplicity of $1$ is equal to $8 + \dim \Eig_{\frac{1}{2}}(a)$; 
\item
the multiplicity of $2$ is equal to $\dim \Eig_1(a) - 4$;
\item
the multiplicity of any other $\lambda$ is equal to 
$\dim \Eig_{\frac{\lambda}{2}}(a)$.
\end{enumerate}
\end{cor}

Beware that if $\Eig_1(a)$ is 4-dimensional (i.e., if $\Eig_1 (a)$
equals $\llangle a, i_{n-1}\rrangle$, 
then $2$ is not an eigenvalue of $(\alpha a, \beta a)$.

\begin{proof}
Note that $1 - 2\norm{\alpha \times \beta} = 0$ 
and $1 + 2\norm{\alpha \times \beta} = 2$,
so parts (c) and (e) of Proposition
\ref{prop:eig-alpha} describe $\Eig_0(\alpha a, \beta a)$.
The analysis of the other eigenvalues follows from the other parts
of Proposition \ref{prop:eig-alpha}.
\end{proof}

We end this section by considering the case when $\alpha \times \beta = 0$;
this is excluded in Proposition \ref{prop:eig-alpha},
Corollary \ref{cor:gamma-not-0}, and Corollary \ref{cor:gamma-max}.

\begin{prop}
\label{prop:gamma=0}
Let $a$ belong to $\C_{n-1}^\perp$, and let $\alpha$ and $\beta$
belong to $\C_{n-1}$.
Suppose that $\alpha \times \beta = 0$
(equivalently, $\alpha$ and $\beta$ are linearly dependent).
Then $\Eig_\lambda(a) \times \Eig_\lambda(a)$ is contained in
$\Eig_\lambda(\alpha a, \beta a)$.
In particular, the eigenvalues of $(\alpha a, \beta a)$ are the
same as the eigenvalues of $a$, but the multiplicities are doubled.
\end{prop}

\begin{proof}
This follows immediately from the formula in Lemma \ref{lem:mult-alpha}.
\end{proof}


\section{Eigentheory of $A_4$}
\label{sctn:A4-eig}

In this section we will completely describe the eigentheory of 
every element of $\C_4^\perp$.  
The eigentheory of an arbitrary element of $A_4$ can then
be described with
Proposition \ref{prop:reduce-C-perp}.

\begin{prop}
\label{prop:A4-lin-dep}
Let $a$ be an imaginary element of $A_3$.
Then $(a\cos \theta, a\sin \theta)$
is alternative in $A_4$.
\end{prop}

In other words, $1$ is the only eigenvalue of $(a,b)$ if
$a$ and $b$ are imaginary and linearly dependent elements of $A_3$.

\begin{proof}
As explained in Section \ref{subsctn:oct}, we may assume that
$a$ is orthogonal to $\C_3$.  Using that $A_3$ is alternative,
the result is a special case of Proposition \ref{prop:gamma=0}.
\end{proof}

Having dispensed with the linearly dependent case, we will now focus
our attention on elements $(a,b)$ of $A_4$ such that $a$ and $b$
are imaginary and linearly independent.

\begin{thm}
\label{thm:A4}
Let $a$ and $b$ be imaginary linearly independent elements of $A_3$
such that $\norm{a}^2 + \norm{b}^2 = 1$ 
(so that $(a,b)$ is a unit vector in $A_4$).
Set $c = a \times b/\norm{a \times b}$.  Then:
\begin{enumerate}[(a)]
\item
$\llangle a, b \rrangle \times \llangle a, b \rrangle$
is contained in the $1$-eigenspace of $(a,b)$;
\item
$\{ (x, - c x)\ |\ x \in \llangle a, b\rrangle^\perp \}$ is contained
in the $(1 + 2\norm{a \times b})$-eigenspace of $(a,b)$;
\item
$\{ (x, c x)\ |\ x \in \llangle a, b\rrangle^\perp \}$ is contained
in the $(1 - 2\norm{a \times b})$-eigenspace of $(a,b)$.
\end{enumerate}
\end{thm}

By counting dimensions, it is straightforward to check that $A_4$
is the direct sum of the subspaces listed in the theorem.  Thus,
the theorem completely describes the eigentheory of $(a,b)$.

Note that the definition of $c$ makes sense because $a \times b$
is always non-zero when $a$ and $b$ are imaginary and linearly
independent.

\begin{proof}
The element $a \times b$ is a non-zero imaginary element of $A_3$.
As explained in Section \ref{subsctn:oct}, we may assume that
$a \times b$ is a non-zero scalar multiple of $i_3$.  Then 
Lemma \ref{lem:cross-product-ortho} implies that $a$ and $b$
belong to $\C_3^\perp$.
Now Proposition \ref{prop:eig-alpha} applies.  
\end{proof}

Another approach to Theorem \ref{thm:A4} is to compute directly
using octonionic arithmetic that
for $x$ in $\llangle a, b \rrangle^\perp$, \[
(a,b) \cdot (a,b) (x, \pm c x ) =
-(1 \mp 2\norm{a \times b})(x, \pm cx).
\]

\begin{cor}
\label{cor:A4-spec}
Let $a$ and $b$ be imaginary linearly independent elements of $A_3$,
and let $\theta$ be the angle between $a$ and $b$.
The eigenvalues of $(a,b)$ are
\[
1,\ 1 + \frac{2 \norm{a} \norm{b} \sin \theta}{\norm{a}^2 + \norm{b}^2},\
1 - \frac{2 \norm{a} \norm{b} \sin \theta}{\norm{a}^2 + \norm{b}^2}.
\]
The multiplicities are $8$, $4$, and $4$ respectively.
\end{cor}

\begin{proof}
See the proof of Corollary \ref{cor:gamma-not-0} 
to reduce to the case in which $(a,b)$ is a unit vector.
Then apply Theorem \ref{thm:A4}.  One also needs
Lemma \ref{lem:cross} to compute the norm of the cross-product;
note that the hypothesis of this lemma is satisfied because $A_3$
is alternative.
\end{proof}

Properly interpreted, the corollary is also valid when
$a$ and $b$ are linearly dependent.  In this case, $\sin \theta = 0$,
and all three eigenvalues are equal to 1.  This agrees with
Proposition \ref{prop:A4-lin-dep}.

Recall 
that $(a,b)$ is a zero-divisor in $A_4$
if and only if $a$ and $b$ are orthogonal imaginary elements of $A_3$ that
have the same norm
\cite[Cor.~2.14]{M1} \cite[Prop.~12.1]{DDD}.

\begin{prop}
\label{prop:A4-zd-spec}
Let $a$ and $b$ be orthogonal, imaginary, non-zero elements of $A_3$ such that
$\norm{a}=\norm{b}$.
Then the eigenvalues of $(a,b)$ in $A_4$ are $0$, $1$, and $2$
with multiplicities $4$, $8$, and $4$ respectively.
Moreover,
\begin{enumerate}[(a)]
\item
$\Eig_0(a,b) = \{(x,-ab \cdot x/\norm{ab}):x\in \llangle a,b\rrangle^\perp \}$.
\item
$\Eig_1(a,b) = \llangle a,b\rrangle \times \llangle a,b\rrangle$.
\item
$\Eig_2(a,b) = \{(x, ab \cdot x/\norm{ab}):x\in \llangle a,b\rrangle^\perp \}$.
\end{enumerate}
\end{prop}

\begin{proof}
This is a special case of Theorem \ref{thm:A4}.  Note that
$ab$ is already imaginary because $a$ and $b$ are orthogonal;
therefore $a \times b = -ab$.
\end{proof}

See \cite[Section 4]{MG} for a generic example of the computation in
Proposition \ref{prop:A4-zd-spec}.

\section{Further Results}

We now establish precisely which real numbers
occur as eigenvalues in Cayley-Dickson algebras.
Recall from \cite[Prop.~9.10]{DDD} that if $a$ belongs to $A_n$,
then the dimension of 
$\Eig_0(a)$ is at most $2^n - 4n + 4$, and 
this bound is sharp.  

\begin{defn}
A \mdfn{top-dimensional zero-divisor} of $A_n$ is a 
zero-divisor whose $0$-eigenspace has dimension $2^n - 4n + 4$.
\end{defn}

\begin{thm}
\label{thm:spec-top}
Let $a$ be a top-dimensional zero-divisor in $A_n$, where $n\geq 3$.
Then the eigenvalues of $a$ are $0$ or $2^k$, where $0 \leq k \leq n-3$.
Moreover,
\begin{enumerate}[(a)]
\item
the multiplicity of $0$ is $2^n - 4n + 4$;
\item
the multiplicity of $1$ is $8$;
\item
the multiplicity of all other eigenvalues is $4$.
\end{enumerate}
\end{thm}

\begin{proof}
The proof is by induction, using Corollary \ref{cor:gamma-max}.
The base case $n = 3$ follows from the fact that $A_3$ is alternative.
The base case $n = 4$ follows from Proposition \ref{prop:A4-zd-spec}.

For the induction step,
recall from \cite[Prop.~15.6]{DDD} 
(see also \cite{DDDD})
that every unit length
top-dimensional zero-divisor
of $A_n$ is of the form $\frac{1}{\sqrt{2}}(a, \pm i_{n-1} a)$, where $a$
is a unit length top-dimensional zero-divisor of $A_{n-1}$.
Finally, apply Corollary \ref{cor:gamma-max}.
\end{proof}

\begin{thm}
\label{thm:eig-exist}
Let $n \geq 4$, and let $\lambda$ be any real number in the
interval $[0,2^{n-3}]$.  There exists an element of $A_n$
that possesses $\lambda$ as an eigenvalue.
\end{thm}

\begin{proof}
From Theorem \ref{thm:spec-top}, there exists an element $a$
in $\C_n^\perp$ that possesses both $0$ and $2^{n-3}$ as an eigenvalue.
Proposition \ref{prop:reduce-C-perp} shows that $a \cos \theta + \sin \theta$
is an element that possesses $\sin^2 \theta$ as an eigenvalue.  This
takes care of the case when $\lambda \leq 1$.

Now suppose that $\lambda \geq 1$.  There exists a value of $\theta$
for which $\sin^2\theta + 2^{n-3}\cos^2\theta = \lambda$.  
Proposition \ref{prop:reduce-C-perp} shows that $a \cos \theta + \sin \theta$
is an element that possesses $\lambda$ as an eigenvalue.
\end{proof}

%
%

\section{Some questions for further study}
\label{sctn:ques}

\begin{ques}
Relate the minimum eigenvalue of $(b,c)$ to the minimum eigenvalues
of $b$ and $c$.
\end{ques}

One might hope for an inequality for $\lambda^-_a$ similar
to the inequality given in Proposition \ref{prop:eig-double}.
However, beware that $(b,c)$ can be a zero-divisor, even if neither
$b$ nor $c$ are zero-divisors.  In other words, $\lambda^-_{(b,c)}$
can equal zero even if $\lambda^-_b$ and $\lambda^-_c$ are both non-zero.

\begin{ques}
Let $a$ belong to $A_n$.  Show that the dimension of $\Eig_1(a)$
cannot equal $2^n - 12$ or $2^n - 4$.
Show that all other multiples of 4 do occur.
\end{ques}

This guess is supported by computer calculations.  It is easy to
see that $2^n-4$ cannot be the dimension of $\Eig_1(a)$.  Just use
Proposition \ref{prop:eig-sum}.

Computer calculations indicate that the element
\[
\big( (0, t), (t + it, 1 + i + j) \big)
\]
of $A_5$ possesses $1$ as an eigenvalue, and the multiplicity is 4.  
See Section \ref{subsctn:oct} for an explanation of the notation.

\begin{ques}
Fix $n$.  Describe the space of all possible spectra of elements
in $A_n$.
\end{ques}

Results such as Theorem \ref{thm:spec-top}
suggest that the answer is complicated.  We don't even have a guess.
A possibly easier question is the following.

\begin{ques}
Fix $n$.  Describe the space of all possible spectra of zero-divisors
in $A_n$.
\end{ques}

\begin{ques}
Study the characteristic polynomial of elements of $A_n$.
\end{ques}


\bibliographystyle{amsalpha}

\end{document}